\documentclass[11pt,reqno]{amsart}

\usepackage[top=1.5in, bottom=2in, left=1in, right=1in]{geometry}

\usepackage{amsthm, amssymb, amstext, amscd, amsfonts, amsxtra, latexsym, amsmath, comment, tikz,
	mathrsfs, mathtools, esint, stmaryrd, cancel, mathtools}
\usepackage{color}
\usepackage{fancybox}
\usepackage[english]{babel}
\usepackage[latin1]{inputenc}
\usepackage{enumitem}
\usepackage{mdframed}
\usepackage{stackengine}
\usepackage{scalerel}
\usepackage{bbm}
\usepackage{scalerel,stackengine}
\usepackage{graphicx} 
\usetikzlibrary{calc}

\stackMath
\newcommand\reallywidehat[1]{%
	\savestack{\tmpbox}{\stretchto{%
			\scaleto{%
				\scalerel*[\widthof{\ensuremath{#1}}]{\kern-.6pt\bigwedge\kern-.6pt}%
				{\rule[-\textheight/2]{1ex}{\textheight}}
			}{\textheight}%
		}{0.5ex}}%
	\stackon[1pt]{#1}{\tmpbox}%
}
\DeclarePairedDelimiter{\floor}{\lfloor}{\rfloor}

\newlength{\parenheight}
\newlength{\parendepth}
\newlength{\parendrop}

\newcommand{\paren}[4]{%
	\settoheight{\parenheight}{\(#4 #2\)}%
	\settodepth{\parendepth}{\(#4 #2\)}
	\addtolength{\parendepth}{.5ex}
	\addtolength{\parenheight}{-.5ex}
	\addtolength{\parenheight}{\parendepth}
	\addtolength{\parendepth}{-.5\parenheight}
	\setlength{\parendrop}{-.5\parenheight}
	\addtolength{\parendrop}{.5ex}
	\raisebox{-\parendepth}{\(#4
		\left#1%
		\rule[\parendrop]{0pt}{\parenheight}%
		\right.\)}
	#2
	\raisebox{-\parendepth}{\(#4
		\left.%
		\rule[\parendrop]{0pt}{\parenheight}%
		\right#3\)}
}

\def\myleft#1#2\myright#3{%
	\mathchoice{%
		\paren{#1}{#2}{#3}{\displaystyle}%
	}{%
		\paren{#1}{#2}{#3}{\textstyle}%
	}{%
		\paren{#1}{#2}{#3}{\scriptstyle}%
	}{%
		\paren{#1}{#2}{#3}{\scriptscriptstyle}%
	}%
}

\numberwithin{equation}{section}
\newtheorem{theorem}{Theorem}
\numberwithin{theorem}{section}
\newtheorem{lemma}[theorem]{Lemma}
\newtheorem{proposition}[theorem]{Proposition}
\newtheorem{corollary}[theorem]{Corollary}

\newtheorem*{conjecture}{Conjecture}
\theoremstyle{remark}
\newtheorem*{remark}{Remark}

\theoremstyle{definition}
\newtheorem{definition}[theorem]{Definition}

\usepackage{enumitem}
\setlist[enumerate]{leftmargin=*, listparindent=\parindent, parsep=0pt, font=\upshape, label=\alph*)} 
\setlist[itemize]{leftmargin=*} 

\newcommand{\Z}{\mathbb{Z}}
\newcommand{\N}{\mathbb{N}}

\newcommand{\R}{\mathbb{R}}
\newcommand{\C}{\mathbb{C}}

\newcommand{\im}{\operatorname{Im}}

\newcommand{\sgn}{\operatorname{sgn}}

\newcommand{\z}{\mathfrak{z}}
\newcommand{\x}{\mathbbm x}
\newcommand{\y}{\mathbbm y}
\renewcommand{\H}{\mathbb{H}}
\newcommand{\SL}{\text{\rm SL}}

\makeatletter
\newcommand{\vast}{\bBigg@{3.5}}
\newcommand{\Vast}{\bBigg@{4.5}}
\newcommand{\VVast}{\bBigg@{6}}
\makeatother

\renewcommand{\b}[1]{\boldsymbol{#1}}
\newcommand{\lf}{\left\lfloor}
\newcommand{\rf}{\right\rfloor}

\def\H{\mathbb{H}}

\begin{document}

\title[Generating functions of planar polygons]{Generating functions of planar polygons from homological mirror symmetry of elliptic curves}
\author{Kathrin Bringmann}
\author{Jonas Kaszian}
\author{Jie Zhou}

\address{Kathrin Bringmann, University of Cologne, Department of Mathematics and Computer Science, Weyertal 86-90, 50931 Cologne, Germany}
\email{kbringma@math.uni-koeln.de}

\address{Jonas Kaszian, Max-Planck-Institut f\"ur Mathematik, Vivatsgasse 7, 53111 Bonn, Germany}
\email{jonask@mpim-bonn.mpg.de}

\address{Jie Zhou, Yau Mathematical Sciences Center, Tsinghua University, Beijing 100084, P.\,R. China}
\email{jzhou2018@mail.tsinghua.edu.cn}

\maketitle



\setcounter{tocdepth}{4}

\begin{abstract}
We study generating functions
of certain shapes of planar polygons arising from homological mirror symmetry of elliptic curves.
We express these generating functions in terms of rational functions of the Jacobi theta function and Zwegers' mock theta function and determine their (mock) Jacobi properties. We also analyze their special values and singularities, which are of geometric interest as well.
\end{abstract}

\section{Introduction and statement of results}

Elliptic curves provide a fertile ground for
the study of the homological mirror symmetry conjecture \cite{Kon:1994}, which relates interesting
algebraic structures occurring in the symplectic geometry and complex geometry of different manifolds.
They are very simple manifolds that nevertheless exhibit surprisingly rich connections to many fields including Hodge theory, modular forms, and mathematical physics.

Of central importance in this subject are
the generating functions arising from the open Gromov-Witten theory of elliptic curves.
They give the structure constants for the  $A_{\infty}$-structure 
(i.e., the homotopy version of associative algebra structure) in the Fukaya category (whose objects are Lagrangian submanifolds carrying vector bundles over them, and whose morphisms concern
relations among the vector bundles). On the one hand, having a clear understanding of these functions is very useful to verify ideas and conjectures in homological mirror symmetry for elliptic curves and even for more general manifolds.
On the other hand, these functions frequently exhibit transformation properties of mock modular forms and Jacobi forms that are interesting to study on their own. Specifically, they provide natural examples of mock modular forms of higher depth. {\it Mock modular forms} are holomorphic parts of so-called {\it harmonic Maass forms}, which are non-holomorphic generalizations of modular forms. Higher depths forms require additional differential operators.
The generating functions arising in this context are very concrete objects and can be expressed using elementary geometric objects.
By definition they enumerate holomorphic disks on elliptic curves bounded by a given set of Lagrangians, with appropriate weights specified for example by
the area of the holomorphic disks.
Due to the simplicity of the universal cover of the elliptic curve, the
Lagrangians are represented by straight lines on the universal cover,
holomorphic disks are then represented by polygons whose edges lie on these straight lines. This allows the reduction of the
enumeration of these geometrical objects to a combinatorial problem.
The resulting generating functions may then be written down and turn out to be indefinite theta functions
\cite{Pol:2000, Pol:20011, Pol:20012, Pol:2005}, see also \cite{BHLW,HLN}.
In particular, it was found in \cite{Pol:20011} that
the enumeration of triangles yields Jacobi theta functions.
The enumeration of parallelograms \cite{Pol:20011, Pol:20012, Pol:2005} gives the G\"ottsche-Zagier series \cite{Gottsche:1998}, while that of
more general shapes of 4-gons give
the Appell-Lerch sums studied by Kronecker that describe sections of rank two vector bundles on the elliptic curve as shown by \cite{Pol:20013}.
Interestingly, while the former only involves the usual Jacobi theta functions, the latter are related to the
mock theta functions.

Recently there also have been some works considering the genus zero open Gromov-Witten invariants of the quotient of elliptic curves called elliptic orbifolds
\cite{BKR,BRZ,CHKL, CHL,Lau:2015}.
A detailed study of the mock modularity of some generating functions arising from this context was performed in \cite{BKR, BRZ, Lau:2015}.
We remark that the objects studied in the present work differ from those in the above mentioned papers in that the occurring generating functions are different: the former mainly works with fixed Lagrangians, while in the present work
deformations of the Lagrangians are considered as set up originally in \cite{Pol:20011}.

In this paper, we follow the lines in \cite{Pol:20011, Pol:20012, Pol:2005} and study the generating functions arising from the enumeration of particular shapes of 4-gons and 5-gons. The main result of this paper is the following (see \eqref{f3} and \eqref{f4} for the generating functions and Theorem \ref{theoremPart1} and Theorem \ref{theoremPart2} for the mock Jacobi properties).
\begin{theorem}\label{mainthm}
The functions $f_3$ and $f_4$ satisfy mock Jacobi properties.
\end{theorem}

A careful analysis of the modular behavior of the generating functions reveals the global properties of the
Gromov-Witten theory on the geometric side.
Moreover, the study of special values and singularities can be used to detect
what happens in the geometric context, which are otherwise very hard to approach (for example, when the Lagrangians do not intersect transversally).
While the study of these very special shapes are already interesting,
we hope to extend our investigation to include more general shapes of 5-gons and 6-gons in future work.

The paper is organized as follows. In Section 2 we provide some preliminary results and conventions on Jacobi theta functions and mock theta functions of Zwegers.
In Section 3 we review the geometric construction of the generating functions.
We then study the generating functions case by case in Sections 4 to 7.
We conclude with some discussions and a conjecture in the final section.
\section*{Acknowledgments} The research of the first author is supported by the Alfried Krupp Prize for Young University Teachers of the Krupp foundation
and the research of all three authors was supported by the Deutsche Forschungsgemeinschaft (German Research Foundation) under the Collaborative Research Centre / Transregio (CRC/TRR 191) on Symplectic Structures in Geometry, Algebra and Dynamics. The authors thank Chris Jennings-Shaffer for helpful comments on an earlier version of this paper and the referees for their helpful comments.

	\section{Preliminaries}
In this section we recall some modular forms and generalizations thereof, which we require for this paper. Note that we frequently suppress $\tau$ in the notation of functions $f:\C^N\times \mathbb{H}\to \C, (\b{z},\tau)\mapsto f(\b{z})=f(\b{z};\tau)$ if it is viewed as fixed. We write real and imaginary parts as $\tau=u+iv\in \C$, $\b{z}=\b{x}+i\b{y}\in \C^N$ and frequently use $q:=e^{2\pi i \tau}$, $\zeta:=e^{2\pi i z}$, and $\zeta_j:=e^{2\pi i z_j}$ for $j\in\N$. The \emph{Dedekind eta function}
\begin{align*}
\eta(\tau):=q^{\frac1{24}} \prod_{n=1}^{\infty} \left(1-q^n\right)
\end{align*}
is a modular form of weight $\frac12$ with multiplier
\begin{align*}
\nu_{\eta}\left(\begin{smallmatrix}
a & b \\ c & d
\end{smallmatrix}\right):=\begin{cases}
\left(\frac{d}{\lvert c\rvert }\right) e^{\frac{\pi i }{12}\left(\left(a+d\right)c-bd\left(c^2-1\right)-3c\right)} &\quad \text{if }c \text{ is odd},\\
\left(\frac{c}{d }\right) e^{\frac{\pi i }{12}\left(ac\left(1-d^2\right)+d\left(b-c+3\right)-3\right)} &\quad \text{if }c \text{ is even},
\end{cases}
\end{align*} 
which means that for $\left(\begin{smallmatrix}
a&b\\c&d
\end{smallmatrix}\right)\in\SL_2(\Z)$ we have
\begin{equation}\label{etamod}
\eta\left(\tfrac{a\tau+b}{c\tau+d}\right)=\nu_{\eta}\left(\begin{smallmatrix}
a&b\\c&d
\end{smallmatrix}\right)(c\tau+d)^{\frac 12} \eta(\tau).
\end{equation}
The \emph{Jacobi theta function} is defined as 
\begin{align*}
\vartheta(z;\tau) := \sum_{n \in\frac12+\Z} q^{\frac{n^2}{2}} e^{2\pi in\left(z+\frac12\right)}= -iq^{\frac 18} \zeta^{-\frac 12}\prod_{n\geq 1} \left(1-q^n\right)\left(1-\zeta q^{n-1}\right)\left(1-\zeta^{-1}q^n\right)
\end{align*}
We require the following properties of $\vartheta$.
\begin{lemma}\label{thetalemma}
	\begin{enumerate}[leftmargin=*, label={\rm (\arabic*)}]
		\item We have
		\begin{align*}
		\vartheta(-z)=-\vartheta(z).
		\end{align*}
		\item For $\ell,m\in\Z$, we have
		\begin{align*}
		\vartheta\left(z+\ell\tau+m\right)=(-1)^{\ell+m} q^{-\frac{\ell^2}{2}}\zeta^{-\ell}\vartheta\left(z\right).
		\end{align*}
		
		\item We have	
		\begin{equation*}\label{Etquot2}
		\frac{\eta^3}{\vartheta\left(\tfrac12\right)\vartheta\left(\tfrac{\tau}{2}\right)} =-\tfrac{i}{2}q^{\frac14}\vartheta\left(\tfrac{\tau}{2}-\tfrac12\right).
		\end{equation*}			
		
		\item We have for $\left(\begin{smallmatrix}
		a&b\\c&d
		\end{smallmatrix}\right)\in\SL_2(\Z)$
		\begin{equation*}
		\vartheta\left(\tfrac{z}{c\tau+d};\tfrac{a\tau+b}{c\tau+d}\right)=\nu_{\eta}^{3}\left(\begin{smallmatrix}
		a&b\\c&d
		\end{smallmatrix}\right)(c\tau+d)^{\frac 12} e^{\frac{\pi i c z^2}{c\tau+d}}\vartheta(z;\tau).
		\end{equation*}
	\end{enumerate}	
\end{lemma}
\begin{remark}
	Lemma \ref{thetalemma} (2), (4) imply that $\vartheta$ transforms like a Jacobi form of weight $\frac 12$ and index $\frac 12$ for $\SL_2(\Z)$ with multiplier $\nu_{\eta}^3$.
\end{remark}	
Furthermore, we use the following higher-dimensional generalization of Jacobi forms.
\begin{definition}
	Let $f:\C^r\times \H\to \C$ be a meromorphic function with possible poles in $\b{z}\in \C^r$. We call $f$ a \emph{meromorphic Jacobi form} of weight $k$ and index $M\in \frac12\Z^{r\times r}$ for the subgroup $\Gamma\subset\SL_2(\Z)$ if it satisfies for some $a>0$ the growth condition 
	\begin{align*}
	f(\b{z};\tau)e^{-\frac{4\pi}{v}\b{y}^TM\b{y}}\in O\left(e^{av}\right) \quad\text{as }v\to \infty,
	\end{align*}
	for $\b{z}\in \C^r,\b{\ell},\b{m}\in \Z^r$ the elliptic transformation 
	\begin{align*}
	f(\b{z}+\b{\ell}\tau+\b{m})= e^{-4\pi i \b{z}^TM \b{\ell}} q^{-\b{\ell}^TM\b{\ell}} f(\b{z})
	\end{align*}
	and for $\left(\begin{smallmatrix}
	a & b \\ c &d 
	\end{smallmatrix}\right) \in \Gamma$ the modular transformation
	\begin{align*}
	f\left(\tfrac{\b{z}}{c\tau+d};\tfrac{a\tau+b}{c\tau+d}\right)=(c\tau+d)^k e^{\frac{2\pi i c}{c\tau+d}\b{z}^TM \b{z}} f(\b{z};\tau).
	\end{align*}
	Both transformation identities can be modified with some multiplier. If $f$ is holomorphic on all of $\C^r\times \H$ and 
	$f(\b{z};\tau)e^{-\frac{4\pi}{v}\b{y}^TM\b{y}}$ is bounded as $v\to \infty$
	, we call it \emph{holomorphic Jacobi form}.
\end{definition}
We call a meromorphic function $f:\C^r\times \H\to \C$ (with possible poles in the $\b{z}$-variable) a \emph{mock Jacobi form} of weight $k$ and index $M\in \frac12\Z^{r\times r}$ for the subgroup $\Gamma\subset \SL_2(\Z)$ if it can be completed in the sense of \cite{BR, DMZ} to a function that transforms as a Jacobi form of the same weight, index, and subgroup (and possibly multiplier).

Next recall Lemma 2.3 of \cite{BRZ}, which states the following.
\begin{lemma}\label{2.3}
	We have, for $0<y_{1},y_{2}<v$
	\begin{equation*}
	\sum_{n\in \mathbb{Z}} \frac{\zeta_{1}^{n}}{1-\zeta_{2} q^{n}}=-i \eta^{3}\frac{\vartheta(z_{1}+z_{2}) }{\vartheta(z_{1}) \vartheta(z_{2})}.
	\end{equation*}	
	
\end{lemma}

Furthermore, we require the Appell functions 
\begin{equation*}
A\left(z_1,z_2;\tau\right):= e^{\pi i z_1} \sum_{n\in\Z} \frac{(-1)^n q^{\frac{n(n+1)}{2}}e^{2\pi i n z_2}}{1-e^{2\pi i z_1}q^{n}},\quad
\mu\left(z_1,z_2;\tau\right):= \frac{A(z_1,z_2;\tau)}{\vartheta(z_2;\tau)} .
\end{equation*}
We recall some properties of $A$ and $\mu$ that can be easily deduced from Proposition 1.4 of \cite{Zw}. In part (4) we moreover state a consequence of Lemma \ref{lemShift} (2) for $z_0=-z-\tfrac12$, $z_1=z$, and $z_2=z-\frac{\tau}{2}+\frac12$.

\begin{lemma}\label{lemShift}
	Let $z,z_0,z_1,z_2\in \C\setminus (\Z\tau +\Z)$ and $\ell\in\Z$.
	\begin{enumerate}[leftmargin=*, label={\rm (\arabic*)}]
		\item 	We have
		\begin{equation*}
		\mu\left(z_1+\tau,z_2+\tau\right)=\mu\left(z_1,z_2\right),\quad
		A\left(z_1+\ell\tau,z_2+\ell\tau\right)=
		(-1)^\ell q^{-\frac{\ell^2}{2}}\zeta^{-\ell}_2 A\left(z_1,z_2\right).
		\end{equation*}	
		\item Assuming that $z_1+z_0,z_2+z_0\not\in\Z\tau+\Z$ we have
		\begin{equation*}
		\mu(z_1+z_0,z_2+z_0) = \mu(z_1,z_2) +   \frac{i \eta^3\vartheta(z_1+z_2+z_0)\vartheta(z_0)}{\vartheta(z_1)\vartheta(z_2) \vartheta(z_1+z_0)\vartheta(z_2+z_0)}.
		\end{equation*}
		\item We have
		\begin{align*}
		&\mu(z_1,z_2) + q^{-\frac12}\zeta^{-1}_1\zeta_2\mu(z_1+\tau,z_2) = -i q^{-\frac18}\zeta_1^{-\frac12}\zeta_2^{\frac12},\\
		&\mu(-z_1,-z_2) = \mu(z_2,z_1)=-\mu(z_1+1,z_2)=\mu(z_1,z_2),\qquad \mu\left(\tfrac12,\tfrac{\tau}{2}\right) = -\tfrac12q^{\frac18}.
		\end{align*}
		\item We have
		\begin{align*}
		A\left(z,z-\tfrac{\tau}{2}+\tfrac12\right)=-\tfrac12 q^\frac18 \vartheta\left(z-\tfrac{\tau}{2}+\tfrac12\right)+\tfrac12q^{\frac14}\vartheta\left(\tfrac{\tau}{2}-\tfrac12\right)\frac{\vartheta\left(z-\tfrac{\tau}{2}\right)\vartheta\left(z+\tfrac12\right)}{\vartheta(z)}.
		\end{align*}
	\end{enumerate}
\end{lemma}
The function $A$ also has a modular completion i.e., adding a (simpler) non-holomorphic piece yields a function $\widehat{A}$ which transforms like a Jacobi form. To be precise, set
\begin{equation*}
\widehat{A}\left(z_1, z_2; \tau\right):= A\left(z_1, z_2; \tau\right)
+\tfrac{i}{2}
\vartheta\left(z_2; \tau\right)R\left(z_1-z_2; \tau\right)
\end{equation*}
with $R(z; \tau):=\sum_{n\in\frac12+\Z}(\sgn(n)-E((n+\frac{y}{v})\sqrt{2v}))
(-1)^{n-\frac12} q^{-\frac{n^2}{2}} e^{-2\pi inz}$. Here $E(x):= 2\int_{0}^{x} e^{-\pi t^2} dt$ denotes the usual {\it error function}. We also define $\widehat{\mu}(z_1,z_2;\tau):=\frac{\widehat{A}(z_1,z_2;\tau)}{\vartheta(z_2;\tau)}$.

The function $\widehat{A}$ transforms as a Jacobi form of weight one and index $\frac 12\left(\begin{smallmatrix}
-1 & 1\\ 1 & 0
\end{smallmatrix}\right)$ as proven in \cite{Zw}.

\begin{lemma}
	\label{2.4altern}
	\begin{enumerate}[leftmargin=* , label={\rm (\arabic*)}]
	\item \label{Appellt} We have, for $\b{\ell},\b{m}\in\Z^2$,
	\begin{align*}
	\widehat{A} \left( \b{z}+\b{\ell}\tau+\b{m} \right) =(-1)^{ \ell_1 + m_1} e^{2\pi i \left(  \ell_1 -\ell_2 \right)z_1 } e^{-2\pi i \ell_1 z_2} q^{\frac{ \ell_1^2}2 - \ell_1 \ell_2} \widehat{A} \left(\b{z} \right).
	\end{align*}
	\item \label{Amod} We have, for $\left(\begin{smallmatrix}a&b\\c&d\end{smallmatrix}\right)\in\SL_2(\Z)$,
	\begin{equation*}
	\widehat{A} \left( \tfrac{z_1}{c \tau +d}, \tfrac{z_2}{c\tau +d}; \tfrac{a\tau+b}{c\tau+d} \right) = (c\tau+d) e^{\frac{\pi i c\left( - z_1^2 +2z_1z_2\right)}{c\tau+d}} \widehat{A} \left(\b{z}; \tau\right).
	\end{equation*}
	\end{enumerate}
\end{lemma}	
\begin{remark}
	 The function $\widehat{\mu}$ transforms like a Jacobi form of weight $\frac12$ and index $\frac 12\left(\begin{smallmatrix}
	-1 & 1\\ 1 & -1
	\end{smallmatrix}\right)$ (with multiplier).
\end{remark}

Furthermore, we let
\begin{multline}
F(\b z;\tau)\\
:= q^{-\frac18}\zeta_1^{-\frac12} \zeta_2^{\frac12} \zeta_3^{\frac12} \left(\sum_{\b n\in\N_0\times \N^2} + \sum_{\b n\in\N_0\times(-\N)^2}\right) (-1)^{n_1} q^{\frac{n_1(n_1+1)}{2} + n_1n_2 +n_1n_3 + n_2 n_3} \zeta_1^{n_1} \zeta_2^{n_2} \zeta_3^{n_3}.\label{defineF}
\end{multline}
Here and throughout we write components of vectors $\b w\in\C^N$ as $w_1,\ldots,w_N$ and $\zeta_j:=e^{2\pi iz_j}$.

Theorem 1.3 of \cite{BRZ} rewrites $F$ in terms of $\mu$ and $\vartheta$.
\begin{lemma}\label{2.4}
	We have for $0<y_2,y_3<v$
	\begin{align*}
	F(\b z)=i\vartheta(z_1)\mu(z_1,z_2)\mu(z_1,z_3)-\frac{\eta^3\vartheta(z_2+z_3)}{\vartheta(z_2)\vartheta(z_3)}\mu(z_1,z_2+z_3).
	\end{align*}
\end{lemma}

\section{Geometric construction}
\label{secgeometricconstruction}
We now review the construction of the generating functions in consideration following \cite{Pol:2000, Pol:2005}. For this, we fix the lattice $\Lambda=\mathbb{Z} \varrho_1\oplus \mathbb{Z} \varrho_2$ in $\mathbb{R}^2$, where
\begin{align*}
\varrho_1:=
(1, 0)\,,
\quad
\varrho_2:=\left(-\tfrac{1}{2}, \tfrac{\sqrt{3}}{2}\right)\,.
\end{align*}
Furthermore we fix three sets $\mathcal{L}_{j},j\in\{1,2,3\}$ of straight lines defined as follows
\begin{align*}
&\mathcal{L}_1:= \left\{ (t_1,0)+\ell+\R(2 \varrho_1+  \varrho_2): \ell\in  \Lambda\right\}, \qquad
\mathcal{L}_2:= \left\{ (t_2,0)+\ell+\R(-\varrho_1 -2 \varrho_2): \ell\in  \Lambda\right\},\\
&\mathcal{L}_3:= \left\{ (t_3,0)+\ell+\R(- \varrho_1+  \varrho_2): \ell\in  \Lambda\right\}.
\end{align*}
The values $t_{j},j\in\{1,2,3\}$ are chosen such that
none of three lines intersect at a common point\footnote{This condition is usually needed in order to avoid many subtleties in defining the Fakaya category.
Below by studying the generating functions we are able to infer what happens if they do intersect.}.

Consider convex $N$-polygons $\Delta$ bounded by a set of straight lines from $\mathcal{L}_{j},j\in\{1,2,3\}$.
Elementary geometry shows that in the present case we must have $3\leq N\leq 6$.
Denote its set of vertices by $v_{1}, \dots, v_{N}$ in the clockwise order, and the oriented edge from the vertex $v_{k}$ to $v_{k+1}$ by $e_k$ for $k\in\{1,2,\dots N\}$, where we use the convention $v_{N+1}=v_{1}$.
We also denote the area of the $N$-gon $\Delta$ by $\mathrm{area}(\Delta)$ and
the length of the edge $e_{k}$ of $\Delta$ by $|e_{k}(\Delta)|$.
We introduce $N$ real-valued variables $\beta_{k}^*,k\in\{1,2,\dots N\}$, one for each edge $e_{k}$.

We fix one of these convex $N$-gons $\Delta_{0}$ with vertices $V_{1},\dots,V_N$ and edges $E_{1},\dots E_N$ and consider the following summation
\begin{align}\label{GenFunction}
\sum_{\Delta\in S(\Delta_0)} \mathrm{sgn}(\Delta)\,
e^{2\pi i \mathrm{area} (\Delta)w} e^{2\pi i \sum_{k=1}^{N} |e_{k}(\Delta)| \beta_{k}^*}\,,\quad \mathrm{Im}(w)>0,
\end{align}
where
\begin{align*}
S(\Delta_0):=&\left\{\Delta: \Delta \text{ is an } N\text{-gon such that } e_k \text{ and } E_k \text{ are on straight lines in the same set } \mathcal{L}_j, \right.
\\&
\qquad \left.\text{and }v_k-V_k\in \Lambda \text{ for all }k\in\{1,\dots,N\} \right\} /\Lambda,
\end{align*}
and $\Lambda$ acts pointswise on an $N$-gon $\Delta$ and the function $\mathrm{sgn}(\Delta)$ is given by (denoting the $j$-th component of a vertex $v_k$ by $(v_k)_j$)
\begin{align*}
\mathrm{sgn}(\Delta): =
\mathrm{sgn} (  (v_{1})_2 -(v_{N})_2)^{N-1},
\end{align*}
where $\sgn(x):=\frac{\lvert x\rvert}{x}$ for $x\neq0$ and $\sgn(0):=0$. One could replace the function $\mathrm{sgn}(\Delta)$
by $\mathrm{sgn} (  ({v}_{k+1})_2 -({v}_{k})_2)^{N-1} $ for any $k$, which would only possibly change the whole summation by an overall sign.

To simplify the summation in \eqref{GenFunction} we find an explicit description of $S(\Delta_0)$. By translation, we can assume that all of the polygons $\Delta$ share the same vertex, say $v_{1}$, with the reference $N$-gon $\Delta_0$.
Denoting the length of the $k$-th edge $E_{k}$ of the reference $N$-gon $\Delta_0$ by
$\alpha_k$, we can describe $\Delta$ by the oriented length of the sides $n_k+\alpha_k\in \R$ with $n_k\in\Z$ (since the intersections of a straight line $\ell_j\in\mathcal{L}_j$ with the lines in $\mathcal{L}_m$, $m\neq j$ have integer distance from each other). We can omit $n_{N-1}+\alpha_{N-1}$ and $n_{N}+\alpha_N$ since they are determined by $n_1+\alpha_1,\dots, n_{N-2}+\alpha_{N-2}$ (since $e_{N-1}$ and $e_N$ have to be parallel to $E_{N-1}$ and $E_N$, respectively), but we get some conditions encoded in $\psi$ below.

Writing $r:=N-2$, $\beta_{k}:=\mathrm{sign} (e_{k}) \beta_{k}^*$, $\tau:=\frac{2}{\sqrt{3}}w$, one obtains that the generating function \eqref{GenFunction} can be written as
\begin{align*}
\sum_{\b{n}\in \mathbb{Z}^{r}} \,\psi(\b{n}+\b\alpha) \,\mathrm{sgn}(n_{r}+\alpha_r) q^{Q (\b{n}+\b\alpha)} e^{2\pi i B( \b{n}+\b\alpha,\,\b\beta)}\,,
\end{align*}
where
\begin{itemize}
\item
$\b{n}+\b\alpha=(n_1+\alpha_1,\dots n_r+\alpha_r)$ denotes the set of independent parameters for the oriented lengths of $\Delta$;
\item
$B$ is the bilinear form such that the quadratic form $ {\sqrt{3}\over 2} Q(\b{n}+\b\alpha):=\frac12  {\sqrt{3}\over 2}B (\b{n}+\b\alpha,\b{n}+\b\alpha)$ is the area of the corresponding $N$-gon $\Delta$;
\item $\psi(\b{n}+\b\alpha)$ is the characteristic function of the region in $\mathbb{Z}^{r}$
such that $Q(\b{n}+\b\alpha)>0$ and that
$\mathrm{sgn} (e_{k}^T  e_{k+1})$ is the same as $\mathrm{sgn} (E_{k}^T E_{k+1})$ for
$k\in\{1,2,\dots N\}$.
\end{itemize}
 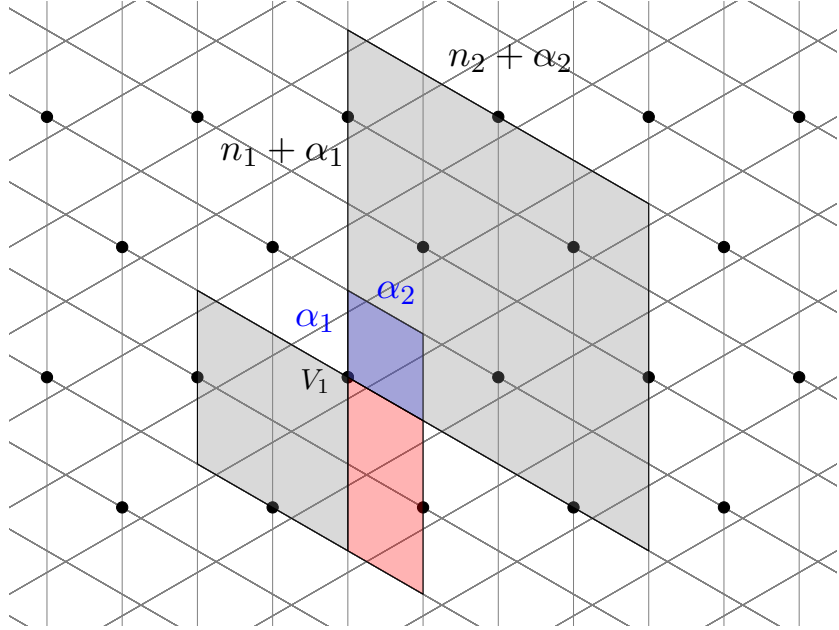
\begin{figure}[ht]
 	\centering
 \begin{tikzpicture}[extended line/.style={shorten >=-#1,shorten <=-#1},
 		extended line/.default=1cm]
 		\coordinate (Origin)   at (0,0);
 		
 		\clip (-4.5,-3.3) rectangle (6.5,5); 

		\foreach \x in {-7,-6,...,7}{
			\foreach \y in {-7,-6,...,7}{
				\draw [style=help lines, extended line = 20cm] (0.5+2*\x-\y,1.7320508*\y) -- (0.5+2*\x-\y+3,1.7320508*\y+1.7320508);
				
			}
		}

 			\foreach \y in {-7,-6,...,21}{
 				\draw [style=help lines, extended line = 20cm] (\y,1.7320508*\y) -- (\y,1.7320508*\y-1);
 				
 			}

 		\foreach \x in {-7,-6,...,7}{
 			\foreach \y in {-7,-6,...,7}{
	 		\draw [style=help lines, extended line = 20cm] (2*\x-\y,1.7320508*\y) -- (2*\x-\y-3,1.7320508*\y+1.7320508);
 				
 			}
 		}
 		
 		\node[draw,circle,inner sep=1.5pt,fill] at (0,0) {};
 		\foreach \x in {-7,-6,...,7}{
 			\foreach \y in {-7,-6,...,7}{
 				\node[draw,circle,inner sep=1.5pt,fill] at (2*\x+\y,\y*1.7320508) {};
 			}
 		}

 		\pgftransformcm{1.7320508/2}{-0.5}{0}{1}{\pgfpoint{0cm}{0cm}}
 		\coordinate (Rhoone) at (0,2/1.7320508);
 		\coordinate (Rhotwo) at (2/1.7320508,0);


 		\filldraw[fill=blue, fill opacity=0.3, draw=black] (Origin)
 		rectangle ($(Rhoone)+(Rhotwo)$);
 		\filldraw[fill=gray, fill opacity=0.3, draw=black] (Origin)
 		rectangle ($4*(Rhoone)+4*(Rhotwo)$);
 		\filldraw[fill=red, fill opacity=0.3, draw=black] (Origin)
 		rectangle ($-2*(Rhoone)+1*(Rhotwo)$);
 		\filldraw[fill=gray, fill opacity=0.3, draw=black] (Origin)
 		rectangle ($-2*(Rhoone)-2*(Rhotwo)$);
 		
 		
 		\node[style={scale=1.3}] at (-0.5,0.5) {\textcolor{blue}{$\alpha_1$}};
 		\node[style={scale=1.3}] at (0.75,1.5) {\textcolor{blue}{$\alpha_2$}};
 		
 		\node[style={scale=1.3}] at (-1, 2.5) {$n_1+\alpha_1$};
 		\node[style={scale=1.3}] at (2.5, 5.5) {$n_2+\alpha_2$};
 		\node[style={scale=1}] at (-0.5, -0.3) {$V_1$};
 		\end{tikzpicture}
 	\caption{The blue parallelogram is $\Delta_0$, and the other parallelograms are shifted such that $v_1=V_1$. The grey parallelograms appear in the summation, but the red one does not. }
 	\label{figure:solving-CVP-bad-basis}
 \end{figure}

An easy inspection shows that the quadratic form $Q$ induced by $B$ has signature $(1,r-1)$. We consider the generating function as a Jacobi form by setting $\b{z}:=\b{\alpha}\tau+\b{\beta}\in\C^{r}$ as an elliptic variable and modify it slightly by multiplying with $q^{-Q(\b{\alpha})} e^{-2\pi i B(\b{\alpha},\b{\beta})}$ to obtain nicer transformation laws and cleaner formulas. Writing $\chi(\b{n}+\b\alpha)=\psi(\b{n}+\b\alpha) \mathrm{sgn}(n_{r}+\alpha_r)$, we define

\begin{align}
\Theta_{Q,\chi}(\b{z};\tau):=&q^{-Q(\b{\alpha})} e^{-2\pi i B(\b{\alpha},\,\b{\beta})}\sum_{\b{n}\in \mathbb{Z}^{r}} \,\psi(\b{n}+\b\alpha) \,\mathrm{sgn}(n_{r}+\alpha_r) q^{Q (\b{n}+\b\alpha)} e^{2\pi i B( \b{n}+\b\alpha,\,\b\beta)}\notag\\
=&
\sum_{\b{n}\in\mathbb{Z}^{r}} \chi\left(\b{n}+\tfrac{\b{y}}{v}\right) q^{Q (\b{n})}e^{2\pi i B\left (\b{n},\,\b{z}\right)}.\label{ThetaDef}
\end{align}
A direct calculation gives the following elliptic transformation.
\begin{lemma}\label{elliptic}
	For $\b{\ell},\b{m}\in\Z^r$ we have
	\begin{align*}
	\Theta_{Q, \chi}(\b{z}+\b{\ell} \tau + \b{m})=q^{-Q(\b{\ell})} e^{-2\pi i B(\b{\ell},\, \b{z})} \Theta_{Q, \chi}(\b{z}).
	\end{align*}
\end{lemma}

%
%

\section{$N=3$: equilateral triangles}

In this section we consider the enumeration of equilateral triangles, for which we have
$$Q(n+\alpha)=\tfrac32(n+\alpha)^2.$$

The enumeration of equilateral triangles leads to the function
\begin{equation*}
f_1(z;\tau):=\sum_{n\in \mathbb{Z}}q^{\frac{3n^2}{2}}\zeta^{3n}.
\end{equation*}


\noindent Note that this is just a renormalized version of one of the Jacobi theta functions, which is a Jacobi form. We compute the elliptic transformation as ($\ell,m\in\Z$)
\begin{align}\label{f1elliptic}
f_1(z+\ell\tau+m)=q^{\frac{3\ell^2}{2}} \zeta^{-3\ell}f_1(z).
\end{align}
Lemma \ref{thetalemma} (4) gives that $f_1$ is a holomorphic Jacobi form of weight $\frac 12$ and index $\frac 32$ on $\Gamma_0(3)\cap \Gamma(2)$. To be more precise, additionally to \eqref{f1elliptic}, we have for $\left(\begin{smallmatrix}
a & b \\  c & d 
\end{smallmatrix}\right)\in \Gamma_0(3)\cap \Gamma(2)$ 
\begin{align*}
f_1\left(\tfrac{z}{c\tau+d};\tfrac{a\tau+b}{c\tau+d}\right)
=\left(\tfrac{3c}{d}\right)e^{\frac{\pi i (d-1)}{4}} (c\tau+d)^{\frac12} e^{\frac{3\pi i c z^2}{c\tau+d}} f_1(z;\tau).
\end{align*}

\section{$N=4$: parallelograms }
In this section, we study the generating function obtained for parallelograms and relate it to the Jacobi theta function.
Following the geometric construction, we have
$$Q(\b{n+\alpha})=3 (n_{1}+\alpha_{1}) (n_{2}+\alpha_{2})$$
and obtain from \eqref{ThetaDef} the generating function for parallelograms as
\begin{align*}
f_{2}(\b{z};\tau)&:=  \sum_{\b n\in \mathbb{Z}^2} \chi_2\left(\b{n}+\tfrac{\b{y}}{v}\right)
q^{3 n_{1}n_{2} } \zeta_{1}^{3n_{2}}\zeta_{2}^{3n_{1}},
\end{align*}
where $\chi_2(\b{x}):=\sgn(x_1)H(x_1x_2)$. Here we define the \emph{Heaviside step function} by $H(x):=1$ for $x> 0$ and $H(x):=0$ for $x\leq 0$.
%

The following elliptic transformation follows directly from Lemma \ref{elliptic}.
\begin{lemma}\label{f2elliptic}
	For $\b{\ell},\b{m}\in\Z^2$ we have
	\begin{align*}
	f_{2}(\b{z}+\b{\ell}\tau+\b m)=q^{-3\ell_1\ell_2}\zeta_1^{-3\ell_2}\zeta_2^{-3\ell_1} f_2\left(\b{z}\right).
	\end{align*}
\end{lemma}
%

We determine the following explicit shape of $f_2$ in terms of the Jacobi theta function.
\begin{proposition}\label{lem1}
	For $y_1,y_2\not\in \Z v$ we have
	\begin{align}\label{f2representation}
	f_2(\b{z};\tau)&=-i
	\eta^{3}(3\tau)
	\frac{ 	\vartheta\left(3z_{1}+3z_{2};3\tau\right) }{
		\vartheta\left(3z_{1} ;3\tau\right) \vartheta\left(3z_{2};3\tau\right)}.
	\end{align}
	The function $f_2$ is a meromorphic Jacobi form of weight one and index $\frac 12\left(\begin{smallmatrix}
	0&3\\3&0
	\end{smallmatrix}\right)$ on $\Gamma_0(3)$. To be more precise, the elliptic transformation law in Lemma \ref{f2elliptic} holds and we have for $\left(\begin{smallmatrix}
	a & b \\ c & d 
	\end{smallmatrix}\right)\in \Gamma_0(3)$
	\begin{align*}
	f_2\left(\tfrac{z_1}{c\tau+d}, \tfrac{z_2}{c\tau+d};\tfrac{a\tau+b}{c\tau+d}\right)=(c\tau+d) e^{\frac{6\pi i cz_1z_2}{c\tau+d}} f_2(z_1,z_2;\tau).
	\end{align*}
\end{proposition}
\begin{proof}
	One can rewrite $f_2$ as
	\begin{align}
	f_{2}(\b{z})& =\left(\sum_{\b n + \frac{\b{y}}{v} > \b 0}-\sum_{\b n + \tfrac{\b{y}}{v} < \b 0}\right)
	q^{3 n_{1}n_{2} } \zeta_{1}^{3n_{2}}\zeta_{2}^{3n_{1}} =  \zeta_{1}^{3 \left(1-\left\lceil \frac{y_2}{v}\right\rceil\right)}   \sum_{n\in \mathbb{Z}}
	\tfrac{ \left(\zeta_{2}^{3} q^{3 \left(1-\left\lceil \frac{y_2}{v}\right\rceil\right)}  \right)^{n}}{1-\zeta_{1}^{3}q^{3n} }.\label{f2rep}
	\end{align}
	Equation \eqref{f2representation} follows for $0<y_1,y_2<v$ using Lemma \ref{2.3} and generalizes to $y_1,y_2\not\in \Z v$ by applying Lemma \ref{f2elliptic} and Lemma \ref{thetalemma} (2). The transformation laws can then deduced from Lemma \ref{thetalemma} (2), (4) and equation \eqref{etamod}. \qedhere	
\end{proof}
The main goal of this section is to study and determine the behavior of $f_2$ at the points of discontinuity. This is done in the following proposition.
\begin{proposition}\label{f2jump}
	Let $y_1\not \in \Z v$ and $y_2\in\Z v$. Then we have for  $x_2-\frac{uy_2}{v}\in\frac13\Z$
	\begin{align}\label{limit1}
	\lim_{\varepsilon\to 0^+} \varepsilon \sum_{\pm} \pm f_2(z_1,z_2\pm i\varepsilon)
	&=
	\tfrac{1}{3\pi }\zeta_1^{-\frac{3y_2}{v}}  ,\qquad
	\lim_{\varepsilon\to 0^+}    \varepsilon f_2(z_1,z_2+ i\varepsilon) = \tfrac{1}{6\pi}\zeta_1^{-\frac{3y_2}{v}}.
	\end{align}
	Moreover for $x_2-\frac{uy_2}{v}\not\in\frac13\Z$, we have
	\begin{align}
	\lim_{\varepsilon\to 0^+} \sum_{\pm} \pm f_2(z_1,z_2\pm i\varepsilon)
	&= 0\label{limit2}
	,\\
	\lim_{\varepsilon\to 0^+}	f_2(z_1,z_2+i\varepsilon;\tau)&=
	 -\frac{i\eta^3(3\tau) \vartheta(3(z_1+z_2);3\tau)}{\vartheta(3z_2;3\tau)}.
	\label{limit3}
	\end{align}

\end{proposition}
\begin{proof}
We first assume that $z_2=x_2\in\R$ and use \eqref{f2rep} to compute,  for $0<\varepsilon<v$,
\begin{align*}
& \sum_{\pm} \pm f_2(z_1,x_2\pm i\varepsilon) \\
&=\!\left(\sum_{n_1+\frac{y_1}{v},\,n_2\geq 0} - \sum_{n_1+\frac{y_1}{v},\,n_2<0}\right) q^{3n_1n_2} \zeta_1^{n_2} e^{6\pi i(x_2+i\varepsilon)n_1} -\!\Vast(\!\sum_{\substack{n_1+\frac{y_1}{v}\geq 0 \\ n_2>0}} \!-\! \sum_{\substack{n_1+\frac{y_1}{v}<0\\n_2\leq 0}}\!\Vast) q^{3n_1n_2} \zeta_1^{n_2} e^{6\pi i(x_2-i\varepsilon)n_1} \\
&=\sum_{\substack{n_1+\frac{y_1}{v}\geq 0 \\ n_2>0}} \!\!\!\!\!\! q^{3n_1n_2}\zeta_1^{3n_2}\left(e^{6\pi in_1(x_2+i\varepsilon)}
-e^{6\pi in_1(x_2-i\varepsilon)}\right)
+ \!\!\!\sum_{\substack{n_1+\frac{y_1}{v}<0\\n_2<0}}\!\!\!\!\!\! q^{3n_1n_2}\zeta_1^{3n_2}\left(-e^{6\pi in_1(x_2+i\varepsilon)}
+ e^{6\pi in_1(x_2-i\varepsilon)}\right)\\
&\quad +
\sum_{n_1+\frac{y_1}{v}\geq 0} e^{6\pi in_1(x_2+i\varepsilon)} + \sum_{n_1+\frac{y_1}{v}<0} e^{6\pi in_1(x_2-i\varepsilon)}.
\end{align*}
The first two sums vanish in the limit $\varepsilon\to 0^+$ since we can exchange limit and summation using Lebesque dominated convergence. The final two terms combine to
\begin{equation*}
e^{-6\pi i\lf\frac{y_1}{v}\rf x_2} \left(\frac{e^{6\pi\lf \frac{y_1}{v}\rf \varepsilon}}{1-e^{6\pi i(x_2+i\varepsilon)}}-\frac{e^{-6\pi \lf\frac{y_1}{v}\rf\varepsilon}}{1-e^{6\pi i(x_2-i\varepsilon)}}\right).
\end{equation*}
From this we obtain \eqref{limit2} and the first claim in \eqref{limit1} in the case that $z_2\in\R$.  In the general case, we write $z_2=x_2+i\ell v=x_2-\ell u+\ell\tau$  for some $\ell\in\Z$ and then employ Lemma \ref{f2elliptic}.

We next compute, using Lemma \ref{f2elliptic} and Proposition \ref{lem1},
\begin{equation*}
f_2(z_1,z_2+i\varepsilon;\tau)=\zeta_1^{-3\ell} f_2(z_1,x_2-\ell u+i\varepsilon;\tau) = -i\zeta_1^{-3\ell} \eta^3(3\tau) \frac{\vartheta(3z_1+3(x_2-\ell u+i\varepsilon);3\tau)}{\vartheta(3z_1;3\tau)\vartheta(3(x_2-\ell u+i\varepsilon);3\tau)}.
\end{equation*}
This directly implies \eqref{limit3}. If $x_2-\ell u \in \frac13 \Z$, we use Lemma \ref{thetalemma} (2) to obtain the second claim in \eqref{limit1}.
\end{proof}

\section{$N=4$: trapezoids}
In this section we study the generating function obtained in \eqref{GenFunction} for trapezoids and relate it to Appell functions. Here we assume without loss of generality that $|\alpha_{2}|<|\alpha_{1}|$ and obtain the quadratic form
	\begin{equation*}
		{\tfrac13}	Q(\b{n}+\b{\alpha})={\tfrac1 2}(n_1+\alpha_{1})^2 -{\tfrac12} (n_2+\alpha_{2})^2.
	\end{equation*}
%
%

\noindent For general values of
 $\alpha_{k}$ and $\beta_{k}$, $k\in\{1,2\}$, we have  $$z_{1}:=\beta_{1}+ \alpha_{1}\tau ,\qquad z_{2}:=\beta_{2}+\alpha_{2}\tau \,,
\quad |\alpha_{2}|<|\alpha_{1}|.$$

The enumeration \eqref{ThetaDef} gives the generating function for trapezoids
\begin{align}\label{f3}
f_3(\b z;\tau):=&\sum_{\b{n}\in\Z^2} \chi_3\left(\b{n}+\tfrac{\b{y}}{v}\right)
q^{\frac32 \left(n_1^2-n_2^2\right)}\zeta_1^{3n_1}\zeta_2^{-3n_2},
\end{align}
where 
\begin{align*}
\chi_3\left({\b x}\right)&:=  \sgn  \left(x_1 \right) H\left(\left| x_1 \right| -\lvert x_2 \rvert \right)H^*(x_1  x_2)
\end{align*}
and $H^*(x):=1$ for $x\geq 0$, $H^*(x):=0$ for $x<0$.

We first again state the elliptic transformation law of $f_3$ that follows by Lemma \ref{elliptic}.
\begin{lemma}\label{lem2}
	We have for $\b \ell, \b m \in\mathbb Z^2$ and $\b z \in\C^2$ \begin{align*}
	f_3(\b z+\b \ell \tau+\b m)&=q^{-\frac32\left(\ell_1^2-\ell_2^2\right)}\zeta_1^{-3\ell_1}\zeta_{2}^{3\ell_2}f_3(\b z).
	\end{align*}
\end{lemma}
We observe the following connection of $f_3$ to Appell functions for generic values of $\b{z}$.
\begin{proposition}\label{f3Appell}
	For $y_2,y_1-y_2\not\in \Z v$, we have
	\begin{align*}
	f_3(\b z;\tau)= \left(\zeta_1^{-1}\zeta_2\right)^{3\left(\tfrac12+\floor*{\tfrac{y_2}{v}}\right)} A\left(3\left(z_1-z_2\right), 3z_1-3\lf\tfrac{y_2}{v}\rf\tau -\tfrac{3\tau}2+\tfrac12; 3 \tau
	\right).
	\end{align*}
\end{proposition}
\begin{proof}
	Using the definition of $f_3$, it is not hard to see that 
	\begin{align*}
	f_3(\b z) &= \tfrac12 \sum_{\b n\in\Z^2} \left(\sgn\left(n_1-n_2+\tfrac1v(y_1-y_2)\right)+\sgn\left(n_2+\tfrac{y_2}{v}\right)\right) q^{\frac32\left(n_1^2-n_2^2\right)}\zeta_1^{3n_1}\zeta_2^{-3n_2}\\
	&=\tfrac12\sum_{\b n\in\Z^2} \left(\sgn\left(n_1+\tfrac1v(y_1-y_2)\right)+\sgn\left(n_2+\tfrac{y_2}{v}\right)\right) q^{\frac32\left(n_1^2+2n_1n_2\right)}\zeta_1^{3(n_1+n_2)}\zeta_2^{-3n_2},
	\end{align*}
	changing variables $n_1\mapsto n_1+n_2$. The claimed identity now follows by using that
	\begin{align*}
	\tfrac12 \sum_{n_2\in\Z} \left(\sgn\left(n_1+\tfrac1v(y_1-y_2)\right)+\sgn\left(n_2+\tfrac{y_2}{v}\right)\right)q^{3n_1n_2}\zeta_1^{3n_2}\zeta_2^{-3n_2} = \frac{q^{-3\lf\frac{y_1}{v}\rf n_1}\left(\zeta_1^{-1}\zeta_2\right)^{3\lf\frac{y_2}{v}\rf}}{1-\zeta_1^3\zeta_2^{-3}q^{3n_1}}
	\end{align*}
	and then plugging in the definition of the Appell function.
\end{proof}
To state the (mock) Jacobi properties of $f_3$, define its completion
\begin{align*}
\widehat{f}_3(\b z;\tau):= \left(\zeta_1^{-1}\zeta_2\right)^{3\left(\tfrac12+\floor*{\tfrac{y_2}{v}}\right)} \widehat{A}\left(3\left(z_1-z_2\right), 3z_1-3\lf\tfrac{y_2}{v}\rf\tau -\tfrac{3\tau}2+\tfrac12; 3 \tau
\right).
\end{align*}
\begin{theorem}\label{theoremPart1}
	The function $f_3$ is a mock Jacobi form of weight one and index $\frac 12\left(\begin{smallmatrix}
	3&0\\0&-3
	\end{smallmatrix}\right)$ for ${\Gamma_0(3)\cap \Gamma(2)}$. To be more precise, we have for $\left(\begin{smallmatrix}
	a & b \\ c &d 
	\end{smallmatrix}\right) \in \Gamma_0(3)\cap \Gamma(2)$
	\begin{align*}
	\widehat{f}_3\left(\tfrac{z_1}{c\tau+d},\tfrac{z_2}{c\tau+d};\tfrac{a\tau+b}{c\tau+d}\right)=(c\tau+d) e^{\frac{\pi i c}{c\tau+d}\left(3z_1^2-3z_2^2\right)} \widehat{f}_3(z_1,z_2;\tau)
	\end{align*}
	and for $\b{\ell},\b{m}\in\Z^2$
	\begin{align*}
	\widehat{f}_3(\b{z}+\b{\ell}\tau+\b{m})= q^{-\frac{3}{2}\left(\ell_1^2-\ell_2^2\right)} \zeta_1^{-3\ell_1}\zeta_2^{3\ell_2}
	\widehat{f}_3(\b{z}).
	\end{align*}
\end{theorem}
\begin{proof}The elliptic and modular properties of the completion $\widehat{f}_3$ can be deduced from those of $\widehat{A}$ after shifting away  $-3\lf\tfrac{y_2}{v}\rf\tau$.
\end{proof}
We next determine the behavior of $f_3$ at the singularities.
\begin{proposition}\label{limlimlemma}
	Assume that $y_1\not\in \Z v$. If $y_2\in \Z v$, then we have
	\begin{multline*}
	\lim_{\varepsilon\to 0^+}f_3\left(z_1,z_2+i\varepsilon;\tau\right)=
	 q^{-\frac{3y_2}{2v}\left(\frac{y_2}{v}+1\right)}\zeta_1^{-\frac32}\zeta_2^{\frac{3y_2}{v}+\frac32}\vartheta\left(3z_1-\tfrac{3\tau}{2}+\tfrac12; 3\tau\right)\mu\left(3z_2-\tfrac{3y_2}{v}\tau+\tfrac12, \tfrac{3\tau}{2};3\tau\right)\\
	-i \left(\zeta_1^{-1}\zeta_2\right)^\frac32 \frac{\eta^3(3\tau)\vartheta(3(z_1-z_2)-\frac{3\tau}{2};3\tau)\vartheta(3z_1+\frac12;3\tau)}{\vartheta(3(z_1-z_2);3\tau)\vartheta(3z_2+\frac12;3\tau)\vartheta(\frac{3\tau}{2};3\tau)}.
	\end{multline*}
In particular, for $z_2=\ell\tau +m$ with $\ell,m\in\Z$ we have
	\begin{align*}\lim_{\varepsilon\to 0^+}f_3\left(z_1,\ell\tau +m+i\varepsilon;\tau\right)&=-\frac12q^{\frac{3\ell^2}{2}+\frac38}\zeta_1^{-\frac32}\vartheta\left(3z_1-\tfrac{3\tau}{2}+\tfrac12;3\tau\right)\\
	&\quad\, -\frac12q^{\frac{3\ell^2}{2}+\frac34}\zeta_1^{-\frac32} \frac{\vartheta\left(\frac{3\tau}{2}-\frac12;3\tau\right)\vartheta\left(3z_1-\frac{3\tau}{2};3\tau\right)\vartheta\left(3z_1+\frac12;3\tau\right)}{\vartheta(3z_1;3\tau)}.
	\end{align*}
\end{proposition} 
\begin{proof}
	We first assume $z_2=x_2\in\R$ and plug in Proposition \ref{f3Appell} to obtain for $y_1\not \in \Z v$ and $\varepsilon>0$,
	\[
	f_3\left(z_1, x_2+i\varepsilon;\tau\right)=e^{6\pi i\left(-z_1+x_2+ i\varepsilon\right)\left(\frac12+\floor*{\frac{\varepsilon}{v}}\right)} A\left(3\left(z_1-x_2- i\varepsilon\right),3z_1-3\floor*{\tfrac{\varepsilon}{v}}\tau-\tfrac{3\tau}{2}+\tfrac12;3\tau\right)
	\]
	and thus
\begin{equation*}
\lim_{\varepsilon\to 0^+}f_3\left(z_1,x_2+i\varepsilon;\tau\right)=\zeta_1^{-\frac32}e^{3\pi i x_2}
 A\left(3z_1-3x_2,3z_1-\tfrac{3\tau}{2}+\tfrac12;3\tau\right).
\end{equation*}

To compute the right-hand side, we rewrite $A(z-x_2 ,z-\tfrac{\tau}{2}+\tfrac12)$, using Lemma \ref{lemShift} (2)
with $z_1=z-x_2 $, $z_2=z-\tfrac{\tau}{2}+\tfrac{1}{2}$, and $z_0=-z-\frac12$, as
\begin{equation*}
A\left(z-x_2,z-\tfrac{\tau}{2}+\tfrac12\right) = \vartheta\left(z-\tfrac{\tau}{2}+\tfrac12\right) \mu\left(-x_2-\tfrac12,-\tfrac{\tau}{2}\right) - \frac{i\eta^3\vartheta\left(z-x_2-\tfrac{\tau}{2}\right)\vartheta\left(-z-\tfrac12\right)}{\vartheta(z-x_2)\vartheta\left(-x_2-\tfrac12\right)\vartheta\left(-\tfrac{\tau}{2}\right)}.
\end{equation*}
Using the second identity in Lemma \ref{lemShift} (3) and simplifying the theta quotient using Lemma \ref{thetalemma} (1) and plugging in $z\mapsto 3z_1$, $x_2\mapsto
3x_2$, and $\tau\mapsto3\tau$ gives
	\begin{multline*}\lim_{\varepsilon\to 0^+}f_3\left(z_1,x_2+i\varepsilon;\tau\right)
	=\zeta_1^{-\frac32}e^{3\pi i x_2} \left(\vphantom{\frac{\tfrac12}{\tfrac12}} \vartheta\left(3z_1-\tfrac{3\tau}{2}+\tfrac12;3\tau\right)\mu\left(3x_2+\tfrac12, \tfrac{3\tau}{2};3\tau\right)\right.\\
-\left. \frac{i\eta^3(3\tau)\vartheta\left(3z_1-3x_2-\tfrac{3\tau}{2};3\tau\right)\vartheta\left(3z_1+\tfrac12;3\tau\right)}{\vartheta(3z_1-3x_2;3\tau)\vartheta\left(3x_2+\tfrac12;3\tau\right)\vartheta\left(\tfrac{3\tau}2;3\tau\right)}\right).
\end{multline*}
To finish the proof, we use Lemma \ref{lem2}. We obtain, writing $z_2=x_2-\ell u+\ell \tau$ with $\ell\in\Z$
\begin{multline*}
\lim_{\varepsilon\to 0^+}f_3\left(z_1,z_2+i\varepsilon;\tau\right)=q^{\frac{3\ell^2}{2}}
\lim_{\varepsilon\to 0^+} e^{6\pi i \ell \left(x_2-\ell u+i\varepsilon\right)}  f_3\left(z_1,x_2-\ell u +i\varepsilon;\tau\right)
\\= q^{\frac{3\ell^2}{2}} e^{6\pi i\ell (x_2-\ell u)} \zeta_1^{-\frac32} e^{3\pi i \left(x_2-\ell u\right)} \left(\vphantom{\frac{\tfrac12}{\tfrac12}} \vartheta\left(3z_1-\tfrac{3\tau}{2}+\tfrac12;3\tau\right)\mu\left(3\left(x_2-\ell u\right)+\tfrac12, \tfrac{3\tau}{2};3\tau\right)\right.\\
-\left. \frac{i\eta^3(3\tau)\vartheta\left(3z_1-3\left(x_2-\ell u\right)-\tfrac{3\tau}{2};3\tau\right)\vartheta\left(3z_1+\tfrac12;3\tau\right)}{\vartheta(3z_1-3\left(x_2-\ell u\right);3\tau)\vartheta\left(3\left(x_2-\ell u\right)+\tfrac12;3\tau\right)\vartheta\left(\tfrac{3\tau}2;3\tau\right)}\right).
\end{multline*}
Using Lemma \ref{thetalemma} (2) and simplifying gives the claim.

The simplified expression in the special case $z_2=\ell\tau+m$ with $\ell,m\in\Z$ follows from a straightforward computation using Lemma \ref{thetalemma} (3).
\end{proof}

We next determine the jumping behavior at the points excluded in Proposition \ref{f3Appell}.  Recall that $A(z_1,z_2)$ has poles for $z_1\in\mathbb Z+\mathbb Z\tau$. Note that the right-hand side of Proposition \ref{f3Appell} is continuous for $z_1-z_2\not\in\Z\tau+\Z$ and $y_2\not\in v\Z$. Thus we may take the limit of Proposition \ref{f3Appell} in this case. Next we consider $y_2\in \mathbb Zv$ and determine the jump.

\begin{lemma}\label{L:lim}
	Assume that $y_2\in\mathbb Z v$. Then we have
	\[
	\lim_{\varepsilon\to0^+}\sum_{\pm}\pm f_3\left(z_1,z_2\pm i\varepsilon;\tau\right)=-q^{-\frac{3y_2^2}{2v^2}+\frac38}\zeta_1^{\frac32}\zeta_2^{\frac{3y_2}{v}}\vartheta\left(3z_1+\tfrac{3\tau}{2}+\tfrac12;3\tau\right).
	\]
\end{lemma}
\begin{proof}
	Write $y_2=\ell v$ with $\ell\in\Z$. Then Lemma \ref{lem2} gives
	\[
	f_3\left(z_1,x_2-\ell u+\ell\tau\pm i\varepsilon\right)=q^{\frac{3\ell^2}{2}}e^{6\pi i\ell \left(x_2-\ell u\pm i\varepsilon\right)} f_3\left(z_1,x_2-\ell u\pm i\varepsilon\right).
	\]
	Thus the left-hand side of Lemma \ref{L:lim} becomes
	\begin{equation*}
	q^{-\frac{3\ell^2}{2}}\zeta_2^{3\ell}\lim_{\varepsilon\to 0^+}\sum_{\pm} \pm e^{\pm 6\pi i\ell\varepsilon} f_3(z_1,x_2-\ell u\pm i\varepsilon).
	\end{equation*}
	We then compute
	\begin{align*}
	&\lim_{\varepsilon\to 0^+} \sum_{\pm} \pm e^{\pm 6\pi i\ell\varepsilon} f_3(z_1,x_2-\ell u\pm i\varepsilon;\tau) \\
	&= \left(\sum_{n_1+\frac{y_1}{v},\,n_2\geq0 } - \sum_{n_1+\frac{y_1}{v},\,n_2<0} - \sum_{n_1+\frac{y_1}{v}\geq 0,\,n_2>0} + \sum_{n_1+\frac{y_1}{v}<0,\,n_2\leq 0}\right)  q^{\frac{3n_1^2}{2}+3n_1n_2} \zeta_1^{3n_1} \left(\zeta_1\zeta_2^{-1}\right)^{3n_2}\\
	&=\left(\sum_{n+\frac{y_1}{v}\geq 0} + \sum_{n+\frac{y_1}{v}<0}\right) q^{\frac{3n^2}{2}}\zeta_1^{3n} = \sum_{n\in \Z} q^{\frac{3n^2}{2}} \zeta_1^{3n} = -q^{\frac38}\zeta_1^{\frac32} \vartheta\left(3z_1+\tfrac{3\tau}{2}+\tfrac12;3\tau\right).
	\end{align*}
	This gives the claim.
\end{proof}

\section{$N=5$: pentagons }\label{pentagons}

In this section we study the enumeration of the pentagons and observe that the behaviour at jumps is determined by the function appearing in the enumeration of trapezoids. In this case, we assume $|\alpha_{3}|<\min\{ | \alpha_{1}|, |\alpha_{2}|\}$ and obtain the quadratic form
	\begin{equation*}
	{\tfrac13}Q(\b{n}+\b{\alpha})=   (n_{1} +\alpha_{1})(n_{2} +\alpha_{2})- {\tfrac12} (n_3 +\alpha_{3})^2 ,
	\end{equation*}
For general values of $\beta_{k}$, $k\in\{1,2,3\}$,
the enumeration of pentagons gives
\begin{equation}\label{f4}
f_4(\b z;\tau):=\sum_{\boldsymbol{n}\in \mathbb{Z}^3}
\chi_4\left(\b{n}+\tfrac{\b{y}}{v}\right) q^{ 3 n_1n_2 - \frac{3n_3^2}{2}} \zeta_1^{3n_2}\zeta_2^{3n_1}\zeta_3^{-3n_3}.
\end{equation}
where $$\chi_4(\b x):= H^*(|x_1|-|x_3|)H^*(|x_2|-|x_3|)H^*(x_1x_2).$$
With $S_1:=\{\boldsymbol{x}\in \R^3: x_1,x_2\geq x_3\geq 0\}$ and $S_2:=\{\boldsymbol{x}\in \R^3: -x_1,-x_2\geq x_3\geq 0\}$, we write 
\begin{align*}
f_4(\b z) &= \sum_{\substack{\b n\in\Z^3 \\ \boldsymbol{n}+\frac{\boldsymbol{y}}{v}\in \pm \left(S_1\cup S_2\right) }} q^{3n_1n_2-\frac{3n_3^2}{2}} \zeta_1^{3n_2} \zeta_2^{3n_1} \zeta_3^{-3n_3}.
\end{align*}
If $y_3,y_1-y_3,y_2-y_3\not\in\mathbb Z v$, then we have
\begin{equation*}
f_4(\b z)=g_4(\b z)+g_4\left(-z_1,-z_2,z_3\right),
\end{equation*}
where (with $S_3:=\{\boldsymbol{x}\in \R^3: x_1,x_2> x_3> 0\}$)
\begin{align*}
g_4(\b z;\tau) &:=  \left(\sum_{\substack{\b n\in\Z^3 \\ \boldsymbol{n}+\frac{\boldsymbol{y}}{v}\in S_1 }} + \sum_{\substack{\b n\in \Z^3 \\ \boldsymbol{n}+\frac{\boldsymbol{y}}{v}\in -S_3 }}\right)  q^{3n_1n_2-\frac{3 n_3^2}{2}} \zeta_1^{3n_2} \zeta_2^{3n_1} \zeta_3^{-3n_3}.
\end{align*}

We now want to write $g_4$ as higher depth Appell functions. We start by making the change of variables $n_1\mapsto n_1+n_3$, $n_2\mapsto n_2+n_3$. Then we have, assuming that $y_3,y_1-y_3,y_2-y_3\not\in \mathbb Zv$ and writing $Y_1:=3\lfloor\frac{y_1-y_3}{v}\rfloor$, $Y_2:=3\lfloor\frac{y_2-y_3}{v}\rfloor$, and $Y_3:=3\lfloor\frac{y_3}{v}\rfloor$ 

\begin{align}
&g_4(\b z)=
\left(\sum_{\substack{\b n\in\Z^3 \\ 
	3\b{n}+ \b{Y}\geq 0}}\!\!+\!\!\sum_{\substack{\b n\in\Z^3\\  
	3\b{n}+ \b{Y}< 0 }}\right)  q^{3n_1n_2+3n_1n_3+3n_2n_3+\frac{3 n_3^2}{2}} \zeta_1^{3\left(n_2+n_3\right)}\zeta_2^{3\left(n_1+n_3\right)}\zeta_3^{-3n_3}.\label{g4rep}
\end{align}
Using Lemma \ref{elliptic}, we obtain the following transformation.
\begin{lemma}\label{lemf4}
	Assume that $\b \ell, \b m \in \mathbb{Z}^3$.\\
	{\rm (1)} We have
		\begin{equation*}
		f_4(\b z+\b \ell\tau+\b m) = q^{-3\ell_1\ell_2+\frac{3\ell_3^2}{2}}\zeta_1^{-3\ell_2} \zeta_2^{-3\ell_1}\zeta_3^{3\ell_3}  f_4(\b z).
		\end{equation*}
	{\rm (2)} We have
		\begin{equation*}
		g_4(\b z+\b \ell\tau+\b m) = q^{-3\ell_1\ell_2+\frac{3\ell_3^2}{2}} \zeta_1^{-3\ell_2} \zeta_2^{-3\ell_1}\zeta_3^{3\ell_3}  g_4(\b z).
		\end{equation*}
\end{lemma}
To rewrite $g_4$ in terms of known functions, we let
\begin{align*}
F^{*}(\b z;\tau) := q^{-\frac18}\zeta_1^{-\frac12} \zeta_2^{\frac12} \zeta_3^{\frac12} \left(\sum_{\b{n}\in\N_0^3} + \sum_{\b{n}\in-\N^3}\right) (-1)^{n_1} q^{\frac{n_1(n_1+1)}{2} + n_1n_2 +n_1n_3 + n_2 n_3} \zeta_1^{n_1} \zeta_2^{n_2} \zeta_3^{n_3}.
\end{align*}
Shifting $n_3\mapsto n_3-\lfloor\frac{y_3}{v}\rfloor,n_j\mapsto n_j-\lfloor\frac{y_j-y_3}{v}\rfloor$, $j\in\{1,2\}$ in $F^\ast$ yields the following lemma.
\begin{proposition}\label{lemma} 
	We have  
	\begin{align*}
	&g_4(\b z;\tau) = i  q^{\frac13(Y_1Y_2+Y_1Y_3+Y_2Y_3)+\frac{Y_3^2}{6} +\frac{Y_3}{2}-\frac38} \zeta_1^{-Y_2-Y_3}\zeta_2^{-Y_1-Y_3} \zeta_3^{Y_3-\frac32}\\
	&\ \times F^*\left(3(z_1+z_2-z_3)-\left(Y_1+Y_2+Y_3 \right)\tau-\tfrac{3\tau}{2}+\tfrac12,
	3z_1-\left(Y_1+Y_3\right)\tau,3z_2-\left(Y_2+Y_3\right)\tau;3\tau\right).
	\end{align*}
\end{proposition}

The following lemma states $F^*$ in terms of the $\mu$-function.
\begin{lemma}\label{lem51}
	For $0<y_2,y_3<v$ 
	\begin{align*}
	F^{*}(\b z) = i\vartheta(z_1) \mu(z_1,z_2) \mu(z_1,z_3)+q^{-\frac12}\zeta_1^{-1}\zeta_2\zeta_3\frac{\eta^3\vartheta(z_2+z_3)}{\vartheta(z_2)\vartheta(z_3)}\mu(z_1+\tau,z_2+z_3).
	\end{align*}
\end{lemma}
\begin{proof}
	We may write
	\begin{align*}
	F^{*}(\b z)=F(\b z)+q^{-\frac18}\zeta_1^{-\frac12}\zeta_2^\frac12\zeta_3^\frac12 \left(\sum_{n_2,n_3\geq 0}-\sum_{n_2,n_3<0}\right)q^{n_2 n_3} \zeta_2^{n_2}\zeta_3^{n_3},
	\end{align*}
	where $F$ is defined in \eqref{defineF}. Now, using Lemma \ref{2.3}, we obtain
	\begin{equation}\label{Tquot}
	\left(\sum_{n_2,n_3\geq 0}-\sum_{n_2,n_3<0}\right)q^{n_2 n_3} \zeta_2^{n_2}\zeta_3^{n_3}=-i\eta^3\frac{\vartheta(z_2+z_3)}{\vartheta(z_2)\vartheta(z_3)}.
	\end{equation}
	The claim then follows using Lemma \ref{lemShift} (3) and Lemma \ref{2.4}.
\end{proof}
We define the completion of $f_4$ as 
\begin{align*}
\widehat{f}_4\left(\b z;\tau\right):=&i  q^{-\frac38} \zeta_3^{-\frac32} \sum_{ \pm} \widehat{F}^*\left(3\left(\pm z_1\pm z_2-z_3\right)-\tfrac{3\tau}{2}+\tfrac12,
\pm 3z_1,\pm 3z_2;3\tau\right)
\end{align*}
with
\begin{align*}
\widehat{F}^{*}(\b z;\tau) := i\vartheta(z_1;\tau) \widehat{\mu}(z_1,z_2;\tau) \widehat{\mu}(z_1,z_3;\tau)+q^{-\frac12}\zeta_1^{-1}\zeta_2\zeta_3\frac{\eta^3\vartheta(z_2+z_3;\tau)}{\vartheta(z_2;\tau)\vartheta(z_3;\tau)}\widehat{\mu}(z_1+\tau,z_2+z_3;\tau).
\end{align*}
Combining the previous results of this section gives the following.
\begin{theorem}\label{theoremPart2}
	The function $f_4$ is a sum of products of mock Jacobi forms of weight $\frac 32$ and index $\frac 12\left(\begin{smallmatrix}
	0& 3&0\\3&0&0\\0&0&-3
	\end{smallmatrix}\right)$ for $\Gamma_0(3)\cap \Gamma(2)$. To be more precise, $\widehat{f}_4$ satisfies for $\left(\begin{smallmatrix}
	a & b \\ c &d 
	\end{smallmatrix}\right) \in \Gamma_0(3)\cap \Gamma(2)$
	\begin{align*}
	\widehat{f}_4\left(\tfrac{z_1}{c\tau+d},\tfrac{z_2}{c\tau+d},\tfrac{z_3}{c\tau+d};\tfrac{a\tau+b}{c\tau+d}\right)=\left(\tfrac{3c}{d}\right)e^{\frac{\pi i (1-d)}{4}}(c\tau+d)^\frac32 e^{\frac{\pi i c}{c\tau+d}\left(6z_1z_2-3z_3^2\right)} \widehat{f}_4(z_1,z_2, z_3;\tau)
	\end{align*}
 and for $\b{\ell},\b{m}\in\Z^3$
	\begin{align*}
	\widehat{f}_4(\b{z}+\b{\ell}\tau+\b{m})= q^{-3\ell_1\ell_2+\frac{3}{2}\ell_3^2} \zeta_1^{-3\ell_2} \zeta_2^{-3\ell_1} \zeta_3^{3\ell_3}\widehat{f}_4(\b{z}).
	\end{align*}
\end{theorem}
\begin{proof}
	Considering the identity in Proposition \ref{lemma} for the completed functions, we can shift away all the $Y_j$ in the arguments, which also cancels all occurring $Y_j$ in the factors outside. Then the weight, index, subgroup, and multiplier of the completion of $g_4$ can be deduced from those of $\vartheta$ and $\widehat{A}$. Since the transformation laws are invariant under $\b{z}\mapsto -\b{z}$, this implies that $f_4$ is a mock Jacobi form of the same weight, index, subgroup, and multiplier.
\end{proof}

The jumps of $g_4$ are at $y_3 \in \Z v$, $y_1-y_3 \in \Z v$, and $y_2-y_3 \in \Z v$. We describe them explicitly in the following proposition.

\begin{proposition} \label{g4limit}
	\begin{enumerate}[leftmargin=* , label={\rm (\arabic*)}]
		\item If $y_3\in \Z v$, $y_1-y_3,y_2-y_3 \notin \mathbb{Z}v$, then we have
		\begin{align*}
		\lim_{\varepsilon\to 0^+} \sum_{\pm} \pm g_4(z_1,z_2,z_3\pm i\varepsilon;\tau) &= -iq^{-\frac32\left(\frac{y_3}{v}\right)^2  }\zeta_3^{\frac{3y_3}{v}} \frac{\eta^3(3\tau)\vartheta(3(z_1+z_2);3\tau)}{\vartheta(3z_1;3\tau)\vartheta(3z_2;3\tau)}.
		\end{align*}
		\item If $y_1-y_3  \in \mathbb{Z}v$, $y_3,y_2-y_3 \notin \mathbb{Z}v$, then we have
		\begin{align*}
		\lim_{\varepsilon\to 0^+} \sum_{\pm} \pm g_4(z_1\pm i\varepsilon,z_2,z_3) &= \left(\zeta_1 \zeta_3^{-1}\right)^{3 \frac{y_1-y_3}{v}} q^{-\frac32 \left(\frac{y_1-y_3}{v}\right)^2} f_3(z_1+z_2-z_3,z_2-z_3).
		\end{align*}
		\item If $y_2-y_3 \in \mathbb{Z}v$, $y_3,y_1-y_3 \notin \mathbb{Z}v$, then we have
		\begin{align*}
		\lim_{\varepsilon\to 0^+} \sum_{\pm} \pm g_4(z_1,z_2\pm i\varepsilon,z_3) &= \left(\zeta_2 \zeta_3^{-1}\right)^{3 \frac{y_2-y_3}{v}} q^{-\frac32 \left(\frac{y_2-y_3}{v}\right)^2} f_3(z_1+z_2-z_3,z_1-z_3).
		\end{align*}
	\end{enumerate}
\end{proposition}
\begin{remark}
	We note that the right-hand side of Proposition \ref{g4limit} (1) is meromorphic in $(z_1,z_2)\in\C^2$ whereas the right-hand sides of Proposition \ref{g4limit} (2) and (3) have jumps in $(z_1,z_2)$, which can be seen by using Lemma \ref{L:lim}.
\end{remark}

\begin{proof}[Proof of Proposition \ref{g4limit}]
	We only prove (1) and (2) since part (3) follows analogously.\\
	(1) We first assume that $z_3=x_3\in\R$ and compute for $0<y_1,y_2<v$, using  \eqref{g4rep} and \eqref{Tquot}
	\begin{align}
	&\lim\limits_{\varepsilon \to 0^+} \sum_{\pm} \pm g_4(z_1,z_2,x_3 \pm i \varepsilon;\tau)\notag \\
	&= \left(\sum_{\b n\in\N_0^3} + \sum_{\b n\in-\N^3}\right) q^{3n_1n_2+3n_1n_3+3n_2n_3+ \frac{3 n_3^2}{2}} \zeta_2^{3n_1} \zeta_1^{3n_2} \left(\zeta_1 \zeta_2 e^{-2 \pi i x_3}\right)^{3n_3} \notag\\
	& \hspace{3.0cm} - \left(\sum_{ \b{n} \in \N_0^2\times \N} + \sum_{\b{n} \in -\left(\N^2\times \N_0\right)}\right) q^{3n_1n_2+3n_1n_3+3n_2n_3+ \frac{3 n_3^2}{2}} \zeta_2^{3n_1} \zeta_1^{3n_2} \left(\zeta_1 \zeta_2 e^{-2 \pi i x_3}\right)^{3n_3}\notag\\
	&= \left(\sum_{ \b{n}\in\N_0^2} - \sum_{ \b{n}\in-\N^2 }\right) q^{3n_1n_2} \zeta_2^{3n_1} \zeta_1^{3n_2}= -i \eta^3 (3\tau) \frac{\vartheta(3(z_1+z_2);3 \tau)}{\vartheta(3z_1;3\tau) \vartheta(3 z_2; 3\tau)}\label{g4limeq}.
	\end{align}
	This gives the claim in special case that $y_3=0$.
	
	In general, we have $y_3=\ell_3v$ for some $\ell_3 \in \mathbb{Z}$, thus $z_3 = x_3-\ell_3 u + \ell_3\tau$. Writing for $j \in \{1,2\}$ $z_j=\z_j+\ell_j \tau$ with $0<\im(\z)<v$, we then obtain, using Lemma \ref{lemf4} (2)
	\begin{align*}
	&g\left(z_1, z_2, z_3\pm i\varepsilon\right)=q^{-3\ell_1\ell_2-\frac{3\ell_3^3}{2}}e^{6\pi i(\ell_3z_3-\ell_2\z_1-\ell_1\z_2\pm i \ell_3\varepsilon)}g\left(\z_1, \z_2, x_3-\ell u_3\pm i\varepsilon\right).
	\end{align*}

Combining with  \eqref{g4limeq} gives the claim.

\noindent (2) We begin by computing, for $z_1=x_1+iy_3$, $y_3 \not = 0$, using \eqref{g4rep}
\begin{align*}
&\lim\limits_{\varepsilon \to 0^+} \sum_{\pm} \pm g_4(x_1+i(y_3 \pm  \varepsilon),z_2,z_3) \\&=    \lim\limits_{\varepsilon \to 0^+} \sum_{\pm}\pm   \left(\sum_{\substack{\b n\in\Z^3 \\ \b n + \frac1v (\pm \varepsilon,y_2-y_3,y_3)\geq 0}}  +  \sum_{\substack{\b n\in\Z^3\\  \b n + \frac1v (\pm \varepsilon,y_2-y_3,y_3)< 0 }}\right) \\
&\hspace{7cm} \times q^{3n_1n_2+3n_1n_3+3n_2n_3+\frac{3n_3^2}{2}} \zeta_1^{3\left(n_2+n_3\right)}\zeta_2^{3\left(n_1+n_3\right)}\zeta_3^{-3n_3}
\\&= \left(  \sum_{\substack{n_1\geq 0\\ n_2 + \frac{y_2-y_3}{v} \geq 0\\ n_3+\frac{y_3}{v}\geq 0}}   -  \sum_{\substack{n_1 > 0\\ n_2 + \frac{y_2-y_3}{v} \geq 0\\ n_3+\frac{y_3}{v}\geq 0}}
+  \sum_{\substack{n_1< 0\\ n_2+\frac{y_2-y_3}{v}<0 \\ n_3+\frac{y_3}{v}<0}}   -  \sum_{\substack{n_1\leq 0\\ n_2+\frac{y_2-y_3}{v}<0 \\ n_3+\frac{y_3}{v}<0}}  \right) \\ &\hspace{7cm } \times    q^{3n_1n_2+3n_1n_3+3n_2n_3+\frac{3n_3^2}{2}} \zeta_1^{3\left(n_2+n_3\right)}\zeta_2^{3\left(n_1+n_3\right)}\zeta_3^{-3n_3}
\\&=\left(\sum_{\substack{n_3 + \frac{y_3}{v} \geq 0 \\ n_2+ \frac{y_2-y_3}{v} \geq 0}} - \sum_{\substack{n_3+ \frac{y_3}{v} <0 \\ n_2 + \frac{y_2-y_3}{v} <0} }\right) q^{3n_2n_3 + \frac{3n_3^2}{2}} \zeta_1^{3n_2} \left(\zeta_1\zeta_2 \zeta_3^{-1}\right)^{3n_3}= f_3(z_1+z_2-z_3,z_2-z_3).
\end{align*}
 Now we consider $z_1\in\C$ with $y_1-y_3=\ell v$ for some $\ell\in\Z$. Then $z_1= x_1 - \ell u + iy_3 + \ell \tau$. Lemma \ref{lemf4} (2) gives that
\begin{align*}\label{shift}
g_4(z_1\pm i \varepsilon,z_2,z_3) = \zeta_2^{-3\ell} g_4 (x_1- \ell u + i (y_3 \pm \varepsilon), z_2, z_3).
\end{align*}
Combining this, we may conclude the claim, using Lemma \ref{lem2}.
\end{proof}

From Proposition \ref{g4limit} we immediately obtain the following corollary.
\begin{corollary}\label{cor1}
	\begin{enumerate}[leftmargin=* , label={\rm (\arabic*)}]
		\item
		If $y_3\in  \mathbb{Z}v, y_1-y_3,y_2-y_3 \notin  \mathbb{Z}v$, then we have
		\begin{align*}
		\lim\limits_{\varepsilon \to 0^+} \sum_{\pm} \pm f_4(z_1,z_2,z_3 \pm i \varepsilon) = 0.
		\end{align*}
		\item If $y_1-y_3\in  \mathbb{Z}v, y_3,y_2-y_3,y_1+y_3$, and $y_2+y_3 \notin  \mathbb{Z}v$, then we have
		\begin{align*}
		&\lim_{\varepsilon\to 0^+} \sum_{\pm} \pm f_4(z_1\pm i\varepsilon,z_2,z_3)=\left(\zeta_1\zeta_3^{-1}\right)^{3\frac{y_1-y_3}{v}} q^{-\frac32\left(\frac{y_1-y_3}{v}\right)^2} f_3(z_1+z_2-z_3,z_2-z_3) .
		\end{align*}
		\item If $y_1+y_3\in  \mathbb{Z}v, y_3,y_2-y_3,y_1-y_3,y_2+y_3 \notin  \mathbb{Z}v$, then we have
		\begin{align*}
		&\lim_{\varepsilon\to 0^+} \sum_{\pm} \pm f_4(z_1\pm i\varepsilon,z_2,z_3)= \left(\zeta_1\zeta_3^{-1}\right)^{3\frac{y_1+y_3}{v}} q^{-\frac32\left(\frac{y_1+y_3}{v}\right)^2} f_3(z_1+z_2+z_3,z_2+z_3).
		\end{align*}
		\item If $y_1-y_3,y_1+y_3\in  \mathbb{Z}v, y_3,y_2-y_3,y_2+y_3 \notin  \mathbb{Z}v$, then we have
	\begin{align*}
		&\lim_{\varepsilon\to 0^+} \sum_{\pm} \pm f_4(z_1\pm i\varepsilon,z_2,z_3)=\left(\zeta_1\zeta_3^{-1}\right)^{\frac{3y_1}{v}} q^{-\frac{3}{2v^2}\left(y_1^2+y_3^2\right)} \\
		& \times\left(\left(\zeta_1^{-1}\zeta_3\right)^{\frac{y_3}{2}} q^{\frac{3y_1y_3}{v}} f_3(z_1+z_2-z_3,z_2-z_3) \right.\left.+ \left(\zeta_1\zeta_3^{-1}\right)^{\frac{y_3}{v}} q^{-\frac{3y_1y_3}{v}} f_3(z_1+z_2+z_3,z_2+z_3)\right).
		\end{align*}
	\end{enumerate}
\end{corollary}
\begin{remark}
Since $f_4(\b{z})=f_4(z_2,z_1,z_3)$, one obtains descriptions of jumps analogous to (2), (3), and (4) by exchanging the first and second variable.
\end{remark}

The following lemma computes one-sided limits to the jumps of $g_4$, which are built from of $\mu$-functions and theta-functions. For this define
 	\begin{equation*}
 	T(z;\tau)
	:=
 	 \frac{\vartheta\left(\tfrac{3\tau}{2}+\tfrac12 ; 3\tau\right)\vartheta\left(3z+\tfrac12;3\tau\right)\vartheta\left(3z+\tfrac{3\tau}{2};3\tau\right)}{\vartheta\left(3z+\tfrac{3\tau}{2}+\tfrac12;3\tau\right)}.
 	\end{equation*}
 \begin{lemma}\label{lem:limitsToZero}
 \begin{enumerate}[leftmargin=* , label={\rm (\arabic*)}]
  \item We have
 	\begin{align*}
 	&\lim_{\varepsilon\to 0^+} g_4(z_1,z_2,i\varepsilon;\tau) \\
 	&= - q^{-\frac32\left(\lf\frac{y_1}{v}\rf^2+\lf\frac{y_2}{v}\rf^2\right)-\frac32\left(\lf\frac{y_1}{v}\rf+\lf\frac{y_2}{v}\rf\right)-\frac38}\zeta_1^{3\lf\frac{y_1}{v}\rf} \zeta_2^{3\lf\frac{y_2}{v}\rf} \vartheta\left(3\left(z_1+z_2\right)-\tfrac{3\tau}{2}+\tfrac12;3\tau\right) \\
 	&\hspace{1cm}\times\mu\left(3\left(z_1+z_2\right)-3\lf\tfrac{y_2}{v}\rf\tau-\tfrac{3\tau}{2}+\tfrac12,3z_1;3\tau\right)\mu\left(3\left(z_1+z_2\right)-3\lf\tfrac{y_1}{v}\rf\tau-\tfrac{3\tau}{2}+\tfrac12,3z_2;3\tau\right) \\
 	&\quad- \frac{i\eta^3(3\tau)}{2\vartheta(3z_1;3\tau)\vartheta(3z_2;3\tau)}\left(-\vartheta(3\left(z_1+z_2\right);3\tau) + q^{\frac38}T\left(z_1+z_2\right) \right).
 	\end{align*}

  \item We have
  \begin{align*}
  &\lim_{\varepsilon\to 0^+} g_4(z_3+ i\varepsilon,z_2,z_3;\tau)\\
  &= -\frac{1}{2} q^{-\frac32 \lf\frac{y_2-y_3}{v}\rf^2-\frac32\lf\frac{y_2-y_3}{v}\rf} \zeta_2^{3\lf\frac{y_2-y_3}{v}\rf}\zeta_3^{-3 \lf \frac{y_2-y_3}{v} \rf - \frac32} \vartheta\left(z_2-\tfrac{3\tau}{2}+\tfrac12;3\tau\right)\\
  &\hspace{4.9cm}\times  \mu\left(3z_2-3\lf\tfrac{y_2-y_3}{v}\rf\tau-\tfrac{3\tau}{2}+\tfrac12,3z_3;3\tau\right)\left(-1+q^{\frac38}\frac{T(z_2)}{\vartheta(3z_2;3\tau)}\right)\\
  &\quad-i q^{-\frac32\lf\frac{y_3}{v}\rf-\frac38} \zeta_3^{3\lf \frac{y_2}{v} \rf +\frac32}  \frac{\eta^3(3\tau)\vartheta(3(z_2+z_3);3\tau)}{\vartheta(3z_2;3\tau)\vartheta(3z_3;3\tau)}\mu\left(3z_2+\tfrac{3\tau}{2}+\tfrac12,3\left(z_2+z_3\right)-3\lf\tfrac{y_3}{v}\rf\tau; 3\tau\right).
  \end{align*}
  \item We have
  \begin{align*}
	&\lim_{\varepsilon\to 0^+} g_4(z_1,z_3+ i\varepsilon,z_3;\tau)\\
	&= -\frac{1}{2}q^{-\frac32 \lf\frac{y_1-y_3}{v}\rf^2-\frac32\lf\frac{y_1-y_3}{v}\rf}\zeta_1^{3\lf\frac{y_1-y_3}{v}\rf}\zeta_3^{-3 \lf \frac{y_1-y_3}{v} \rf - \frac32} \vartheta\left(z_1-\tfrac{3\tau}{2}+\tfrac12;3\tau\right)\\
	&\hspace{5cm}\times  \mu\left(3z_1-3\lf\tfrac{y_1-y_3}{v}\rf\tau-\tfrac{3\tau}{2}+\tfrac12,3z_3;3\tau\right)\left(-1+q^{\frac38}\tfrac{T(z_1)}{\vartheta(3z_1;3\tau)}\right)\\
	&\quad-i q^{-\frac32\lf\frac{y_3}{v}\rf-\frac38} \zeta_3^{3\lf \frac{y_3}{v} \rf+\frac32 }  \frac{\eta^3(3\tau)\vartheta(3(z_1+z_3);3\tau)}{\vartheta(3z_1;3\tau)\vartheta(3z_3;3\tau)}\vartheta\left(3z_1+\tfrac{3\tau}{2}+\tfrac12,3\left(z_1+z_3\right)-3\lf\tfrac{y_3}{v}\rf\tau;3\tau\right).
\end{align*}
 \end{enumerate}
 \end{lemma}

\begin{proof}
(1) We use Proposition \ref{lemma} and Lemma \ref{lem51} and simplify the occurring functions using Lemma \ref{thetalemma} (2), Lemma \ref{lemShift} (1) and (4) to conclude the statement after a lengthy calculation.

\noindent (2) The claim follows in a similar way.
\end{proof}

\section{Discussion and open questions}

For the enumeration of $N$-gons with $3\leq N\leq 5$, the explicit computations in the previous sections
exhibit nice formulas for the generating functions in terms of rational functions in the Jacobi theta functions and the $\mu$-function. Using the results in \cite{Zw} about the $\mu$-function,  this tells that the generating functions are actually
mock objects whose modular completions can be easily found.
One geometric consequence of the mock modularity is that the generating functions, originally defined
around $\tau=i\infty$, can be extended to the global moduli of elliptic curves upon modular completion.

The generating function of $6$-gons can be written down similarly according to the geometric construction
reviewed earlier in Section \ref{secgeometricconstruction}. It is essentially given by the following
\begin{align*}
&&f_5(\b z; \tau):=\sum_{\boldsymbol{n}\in \mathbb{Z}^4 \cap D }\sgn\left(n_1+\alpha_{1}-n_3-\alpha_{3}\right)
q^{ 3 n_1n_2-\frac{3}{2}n_3^2-\frac{3}{2}n_4^2} \zeta_1^{3n_2} \zeta_2^{3n_1} \zeta_3^{-3n_3} \zeta_4^{-3n_4},
\end{align*}
where the parameters $\alpha_{k},k\in\{1,2,3,4\}$ satisfy $ \lvert \alpha_3\rvert,\lvert \alpha_4\rvert\leq \min(\lvert \alpha_1\rvert,\lvert \alpha_2\rvert)$ and the region $D$ is
given by
\begin{align*}
D:=&-\b{\alpha}+\left\{\b x\in\R^4: \lvert x_3\rvert,\lvert x_4\rvert\leq \min(\lvert x_1\rvert,\lvert x_2\rvert) \text{ and }x_1x_2, x_1x_3,   x_1x_4\geq 0\right\}.
\end{align*}
Motivated by the studies on homological mirror symmetry \cite{Pol:20012}, we propose the following conjecture.
\begin{conjecture}
The generating function $f_5$ has mock Jacobi properties.
\end{conjecture}
While directly identifying this generating function in terms of Appell functions and theta functions seems to be difficult, it should be possible to determine its mock Jacobi properties using the theory of indefinite theta functions of arbitrary signature. Such an approach could also enable progress on $N$-gons with arbitrary numbers of vertices $N\in\N$, which requires a more uniform geometric setup for $N$-gons.


\end{document}